\documentclass[10pt]{amsart}
\usepackage{amsmath}%
\usepackage{amsfonts}%
\usepackage{amssymb}%
\usepackage{graphicx}
\usepackage{pictexwd,dcpic}
\usepackage{mathrsfs}
\usepackage{booktabs}

\newtheorem{definition}[equation]{Definition}
\newtheorem{theorem}[equation]{Theorem}

\newtheorem{Lemma}[equation]{Lemma}

\newtheorem{corollary}[equation]{Corollary}

\newtheorem{proposition}[equation]{Proposition}

\DeclareMathOperator{\Vol}{Vol}
\DeclareMathOperator{\Int}{int}
\DeclareMathOperator{\Conv}{Conv}
\DeclareMathOperator{\Area}{Area}
\newcommand{\rom}[1]{\uppercase\expandafter{\romannumeral #1\relax}}

\makeatother

\begin{document}
\bibliographystyle{plain}

\title[]{Upper bounds on packing density for circular cylinders with high aspect ratio}
\author[W\"oden Kusner]{W\"oden Kusner\\Department of Mathematics\\ University of Pittsburgh}
\thanks{Research partially supported by NSF grant 1104102.}
%\date{2013}
\maketitle

\begin{abstract}
In the early 1990s, A. Bezdek and W. Kuperberg used a relatively simple argument to show a surprising result:  The maximum packing density of circular cylinders of infinite length in $\mathbb{R}^3$ is exactly $\pi/\sqrt{12}$, the planar packing density of the circle.  This paper modifies their method to prove a bound on the packing density of finite length circular cylinders.  In fact, the maximum packing density for unit radius cylinders of length $t$ in $\mathbb{R}^3$ is bounded above by $\pi/\sqrt{12} + 10/t$.
\end{abstract}

\section{Introduction}
The problem of computing upper bounds for the packing density of a specific body in $\mathbb{R}^3$ can be difficult.  Some known or partially understood non-trivial classes of objects are based on spheres \cite{hales2005proof}, bi-infinite circular cylinders \cite{bezdek1990maximum}, truncated rhombic dodecahedra \cite{bezdek1994remark} and tetrahedra \cite{gravel2011upper}.  This paper proves an upper bound for the packing density of congruent capped circular cylinders in $\mathbb{R}^3$. The methods are elementary, but can be used to prove non-trivial upper bounds for packings by congruent circular cylinders and related objects, as well as the sharp bound for half-infinite circular cylinders.

\subsection{Synopsis}
The density bound of A. Bezdek and W. Kuperberg for bi-infinite cylinders is proved in three steps.  Given a packing of $\mathbb{R}^3$ by congruent bi-infinite cylinders, first decompose space into regions closer to a particular axis than to any other.  Such a region contains the associated cylinder, so density may be determined with respect to a generic region. Finally, this region can be sliced perpendicular to the particular axis and the area of these slices estimated: This area is always large compared to the cross-section of the cylinder. 

In the case of a packing of $\mathbb{R}^3$ by congruent finite-length cylinders, this method fails. The ends of a cylinder may force some slice of a region to have small area. For example, if a cylinder were to abut another, a region in the decomposition might not even wholly contain a cylinder.  To overcome this, one shows that these potentially small area slices are always associated to a small neighborhood of the end of a cylinder. For a packing by cylinders of a relatively high aspect ratio, neighborhoods of the end of a cylinder are relatively rare.  By quantifying the rarity of cylinder ends in a packing, and bounding the error contributed by any particular cylinder end, the upper density bound is obtained.  

\subsection{Objects considered}
Define a \emph{$t$-cylinder} to be a closed solid circular cylinder in $\mathbb{R}^3$ with unit radius and length $t.$ Define a \emph{capped $t$-cylinder} (Figure \ref{fig1}) to be a closed set in $\mathbb{R}^3$ composed of a $t$-cylinder with solid unit hemispherical caps.  A capped $t$-cylinder $C$ decomposes into the $t$-cylinder body $C^0$ and two caps $C^1$ and $C^2\negthinspace.$ The axis of the capped $t$-cylinder $C$ is the line segment of length $t$ forming the axis of $C^0\negthinspace.$  The capped $t$-cylinder $C$ is also the set of points at most 1 unit from its axis.

\subsection{Packings and densities}
A $packing$ of $X\subseteq \mathbb{R}^3$ by capped $t$-cylinders is a countable family $\mathscr{C} = \{C_i\}_{i \in I}$ of congruent capped $t$-cylinders $C_i$ with mutually disjoint interiors  and $C_i \subseteq X$.   For a packing $\mathscr{C}$ of $\mathbb{R}^3$, the $restriction$ of $\mathscr{C}$ to $X\subseteq \mathbb{R}^3$ is defined to be a packing of $\mathbb{R}^3$ by capped $t$-cylinders $\{C_i: C_i \subseteq X\}.$ Let $B(R)$ be the closed ball of radius $R$ centered at $0$. In general, let $B_x(R)$ be the closed ball of radius $R$ centered at $x$.   The \emph{density} $\rho(\mathscr{C}, R, R')$ of a packing $\mathscr{C}$ of $\mathbb{R}^3$ by capped $t$-cylinders with $R\le R'$ is defined as 

$$ \rho(\mathscr{C},R, R') \hspace{2pt}=\hspace{-7pt} \sum_{C_i\subseteq B(R)}\frac{\Vol(C_i)}{\Vol(B(R'))}.$$
 
The \emph{upper density} $\rho^+$ of a packing $\mathscr{C}$ of $\mathbb{R}^3$ by capped $t$-cylinders is defined as 

$$\rho^+(\mathscr{C}) = \limsup_{R\rightarrow \infty} \rho(\mathscr{C},R, R).$$

\begin{figure}[h!]
  \centering
      \includegraphics[width=0.50\textwidth]{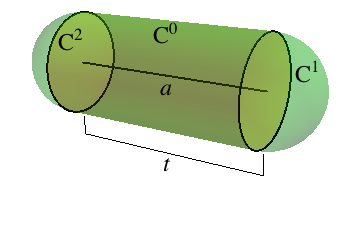}
  \caption{A capped $t$-cylinder with body $C^0$, axis $a$ and caps $C^1$ and $C^2$.}
 \label{fig1}
\end{figure} 

In general, a $packing$ of $X\subseteq \mathbb{R}^3$ by a convex, compact object $K$ is a countable family $\mathscr{K} = \{K_i\}_{i \in I}$ of congruent copies of $K$ with mutually disjoint interiors  and $K_i \subseteq X$.   Restrictions and densities of packings by $K$ are defined in an analogous fashion to those of packings by capped $t$-cylinders.

\section{The Main Results}
Let $t_0 = \frac{4}{3}(\frac{4}{\sqrt{3}}+1)^3 = 48.3266786\dots$ for the remainder of the paper.   This value comes out of the error analysis in Section \ref{part2}.

\begin{theorem}\label{main}  Fix $t \ge 2t_0.$
Fix $R \ge 2/\sqrt{3}.$ Fix a packing $\mathscr{C}$ of $\mathbb{R}^3$ by capped $t$-cylinders. Then  
$$\rho(\mathscr{C},R-2/\sqrt{3},R) \le \frac{t+\frac{4}{3}}{\frac{\sqrt{12}}{\pi}(t-2t_0) + (2t_0)+\frac{4}{3}}.$$
\end{theorem}

This is the content of Sections \ref{setup}, \ref{part1}, \ref{part2}.  Notice that for $t \le 2t_0$, the only upper bound provided is the trivial one.

\begin{corollary}
Fix $t\ge 2t_0.$ The upper density of a packing $\mathscr{C}$ of $\mathbb{R}^3$ by capped $t$-cylinders satisfies the inequality  $$\rho^+(\mathscr{C}) \le \frac{t+\frac{4}{3}}{\frac{\sqrt{12}}{\pi}(t-2t_0) + (2t_0)+\frac{4}{3}}.$$
\end{corollary}
\begin{proof}

Let $V_R$ and $W_R$ be subsets of the index set $I$, with  $V_R = \{i : C_i \subseteq B(R)\}$ and $W_R= \{i : C_i \subseteq B(R-2/\sqrt{3})\}.$   By definition,

$$\rho^+(\mathscr{C}) = \limsup_{R\rightarrow \infty} \left( \sum_{W_R} \frac{\Vol(C_i)}{\Vol(B(R))} \hspace{2pt} + \sum_{V_R\smallsetminus W_R} \hspace{-2pt}\frac{\Vol(C_i)}{\Vol(B(R))} \right) .$$
 
As $R$ grows, the term $\sum_{V_R\smallsetminus W_R} \Vol(C_i)/\Vol(B(R))$ tends to $0$.  Further analysis of the right-hand side yields
$$\rho^+(\mathscr{C}) = \limsup_{R\rightarrow \infty}\rho(\mathscr{C},R-2/\sqrt{3}, R).$$ By Theorem \ref{main}, the stated inequality holds. \end{proof}
\begin{Lemma}\label{nestdensity}
Given a packing of $t$-cylinders with density $\rho$ where $t$ is at least 2, there is a packing of capped $(t-2)$-cylinders with packing density $(\frac{ t - \frac{2}{3} }{t})\cdot\rho.$
\end{Lemma}
\begin{proof}
 \begin{figure}[h!]
  \centering
      \includegraphics[width=0.50\textwidth]{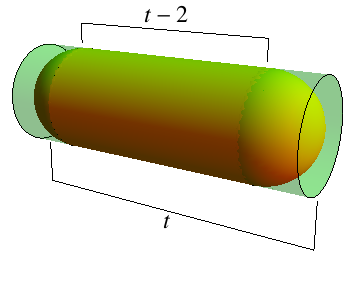}
  \caption{Nesting capped $(t-2)$-cylinders in $t$-cylinders.}
  \label{fig3}
\end{figure} 

From the given packing of $t$-cylinders, construct a packing by capped $(t-2)$-cylinders by nesting as illustrated in Figure \ref{fig3}. By comparing volumes, this packing of capped $(t-2)$-cylinders has the required density.
\end{proof}

\begin{corollary}
Fix $t \ge 2t_0+2$.  The upper density of a packing $\mathscr{Z}$ of $\hspace{2pt} \mathbb{R}^3$ by $t$-cylinders satisfies the inequality

$$ \rho^+(\mathscr{Z}) \le \frac{t}{\frac{\sqrt{12}}{\pi}(t-2-2t_0) + (2t_0)+\frac{4}{3}}.$$
\end{corollary}
\begin{proof} Assume there exists a packing by  $t$-cylinders exceeding the stated bound.
Then Lemma \ref{nestdensity} gives a packing of capped $(t-2)$-cylinders with density greater than 

$$\left(\frac{ t - \frac{2}{3} }{t}\right)\cdot \left(\frac{t}{\frac{\sqrt{12}}{\pi}(t-2-2t_0) + (2t_0)+\frac{4}{3}}\right) =  \left(\frac{t-2+\frac{4}{3}}{\frac{\sqrt{12}}{\pi}(t-2-2t_0) + (2t_0)+\frac{4}{3}}\right). $$

This contradicts the density bound of Theorem \ref{main} for capped $(t-2)$-cylinders.
\end{proof}

 \section{Set Up}\label{setup} 
For the remainder of the paper, fix the notation $\mathscr{C}^*$ to be the restriction of $\mathscr{C}$ to $B(R-2/\sqrt{3})$, indexed by $I^*$.  To bound the density $\rho(\mathscr{C}^*,R-2/\sqrt{3},R)$ for a fixed packing $\mathscr{C}$ and a fixed $R \ge 2/\sqrt{3}$, decompose $B(R)$ into regions $D_i$ with disjoint interiors such that $C_i \subseteq D_i$ for all $i$ in $I^*$.  For such a packing $\mathscr{C}^*$ with fixed $R$, define the \emph{Dirichlet cell} $D_i$ of a capped $t$-cylinder $C_i$ to be the set of points in $B(R)$ no further from the axis $a_i$ of $C_i$ than from any other axis $a_j$ of $C_j$.   

For any point $x$ on axis $a_i$, define a plane $P_x$ normal to $a_i$ and containing $x$.  Define the \emph{Dirichlet slice} $d_x$ be the set $D_i \cap P_x.$   For a fixed Dirichlet slice $d_x$, define $S_x(r)$ to be the circle of radius $r$ centered at $x$ in the plane $P_x$.  Important circles are $S_x(1)$, which coincides with the cross section of the boundary of the cylinder, and $S_x(2/\sqrt{3})$, which circumscribes the regular hexagon in which $S_x(1)$ is inscribed. An $end$ of the capped $t$-cylinder $C_i$ refers to an endpoint of the axis $a_i.$  
 
 \begin{figure}[h!]
  \centering
      \includegraphics[width=0.50\textwidth]{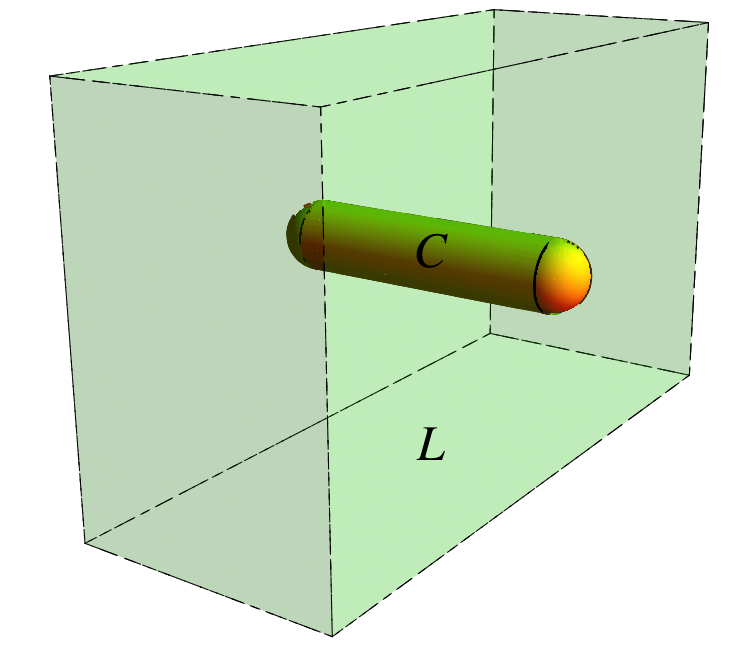}
  \caption{A capped cylinder $C$ and the slab $L$.}
  \label{slab}
\end{figure} 

 Define the \emph{slab} $L_i$ to be the closed region of $\mathbb{R}^3$ bounded by the normal planes to $a_i$ through the endpoints of $a_i$ and containing $C_i^0$ (Figure \ref{slab}). The Dirichlet cell $D_i$ decomposes into the region $D_i^0 = D_i \cap L_i$ containing $C_i^0$ and complementary regions $D_i^1$ and $D_i^2$ containing the caps $C_i^1$ and $C_i^2$ respectively (Figure \ref{Dirdecomp}).  
 
 \begin{figure}[h!]
  \centering
      \includegraphics[width=0.50\textwidth]{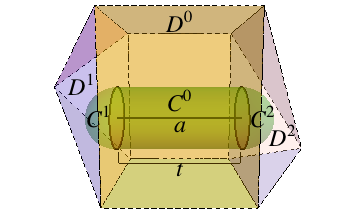}
  \caption{Decomposing a Dirichlet cell.}
  \label{Dirdecomp}
\end{figure} 

Aside from a few degenerate cases, the set of points equidistant from a point $x$ and line segment $a$ in the affine hull of $x$ and $a$ form a $parabolic$ $spline$ (Figure \ref{spline}).  A parabolic spline is a parabolic arc extending in a $C^1$ fashion to rays at the points equidistant to both the point $x$ and an endpoint of the line segment $a$.  Call the points where the parabolic arc meets the rays the Type~\rom{1} points of the curve.  A $parabolic$ $spline$ $cylinder$ is a surface that is the cylinder over a parabolic spline.

   \begin{figure}[h!]
  \centering
      \includegraphics[width=0.50\textwidth]{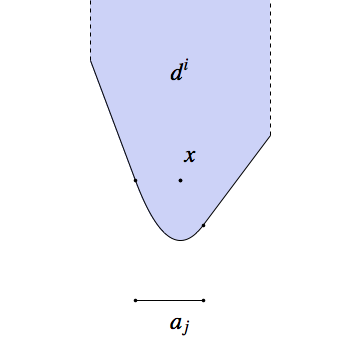}
  \caption{Parabolic spline associated with point $x$ and segment $a_j$.}
  \label{spline}
\end{figure} 

\section{Qualified Points}\label{part1}
\begin{definition}Fix a packing $\mathscr{C}$ of $\mathbb{R}^3$ by capped $t$-cylinders. Fix $R\ge 2/\sqrt{3}$ and restrict to $\mathscr{C}^*$.  A point $x$ on an axis is \emph{qualified} if the Dirichlet slice $d_x$ has area greater than $\sqrt{12}$, the area of the regular hexagon in which $S_x(1)$ is inscribed.
\end{definition}

\begin{proposition}\label{qual}
Fix a packing $\mathscr{C}$ of $\mathbb{R}^3$ by capped $t$-cylinders.  Fix $R\ge 2/\sqrt{3}$ and restrict to $\mathscr{C}^*$.  Let $x$ be a point on an axis $a_i$, where $i$ is a fixed element of $I^*\hspace{-3pt}.$ If $\hspace{2pt} B_x(4/\sqrt{3})$ contains no ends of $\mathcal{C}^*$, then $x$ is qualified.
\end{proposition}

The proof of this proposition will modify the Main Lemma of \cite{bezdek1990maximum}.  A series of lemmas allow for the truncation and rearrangement of the Dirichlet slice.  The goal is to construct from $d_x$ a subset $d_x^{**}$ of $P_x$ with the following properties: 
\begin{itemize} 
\item $d_x^{**}$ contains $S_x(1)$.
\item The boundary of $d_x^{**}$ is composed of line segments and parabolic arcs with apexes touching $S_x(1).$
\item The non-analytic points of the boundary of $d_x^{**}$ lie on $S_x(2/\sqrt{3}).$
\item The area of $d_x^{**}$ is less than the area of $d_x.$
\end{itemize}
Then the computations of \cite[\S6]{bezdek1990maximum} apply.

\begin{Lemma}
If a point $x$ satisfies the conditions of Proposition \ref{qual}, then the Dirichlet slice $d_x$ is a bounded convex planar region, the boundary of which is a simple closed curve consisting of a finite union of parabolic arcs, line segments and circular arcs.  
\end{Lemma}

\begin{proof}
Without loss of generality, fix a point $x$ on $a_i$.  For each $j \ne i$ in $I^*\negthinspace,$ let $d^j$ be the set of points in $P_x$ no further from $a_i$ than from $a_j$.  The Dirichlet slice $d_x$ is the intersection of  $B(R)$ with $d^j$ for all $j \ne i$ in $I^*\negthinspace.$ The boundary of $d^j$ is the set of points in $P_x$ that are equidistant from $a_i$ and $a_j$.  As $P_x$ is perpendicular to $a_i$ at $x$, the boundary of $d^j$ is also the set of points in $P_x$ equidistant from $x$ and $a_j$.

This is the intersection of the plane $P_x$ with the set of points in $\mathbb{R}^3$ equidistant from $x$ and $a_j$.  The set of points  in $\mathbb{R}^3$ equidistant from $x$ and $a_j$ is a parabolic spline cylinder perpendicular to the affine hull of $x$ and $a_j$.  Therefore the set of points equidistant from $x$ and $a_j$ in $P_x$ is also a parabolic spline, with $x$ on the convex side.

In the degenerate cases where $x$ is in the affine hull of $a_j$ or $P_x$ is parallel to $a_j$, the set of points equidistant from $x$ and $a_j$ in $P_x$ are lines or is empty.

The region $d_x$ is clearly bounded as it is contained in $B(R).$
The point $x$ lies in the convex side of the parabolic spline so each region $d^j$ is convex.  The set $B(R)$ contains $x$ and is convex, so $d_x$ is convex.
This is a finite intersection of regions bounded by parabolic arcs, lines and a circle, so the rest of the lemma follows.
\end{proof}

To apply the results of \cite{bezdek1990maximum}, the non-analytic points of the boundary of the Dirichlet slice $d_x$ must be controlled.  From the construction of $d_x$ as a finite intersection, the non-analytic points of the boundary of $d_x$ fall into three non-disjoint classes of points: the Type~\rom{1} points of a parabolic spline that forms a boundary arc of $d_x$, Type~\rom{2} points defined to be points on the boundary of $d_x$ that are also on the boundary of $B(R)$, and Type~\rom{3} points, defined to be points on the boundary of $d_x$ that are equidistant from three or more axes. Type III points are the points on the boundary of $d_x$ where the parabolic spline boundaries of various $d^j$ intersect.

\begin{Lemma}\label{key1}
If a point $x$ satisfies the conditions of Proposition \ref{qual}, then no non-analytic points of the boundary of $d_x$ are in $\Int(\Conv(S_x(2/\sqrt{3})))$, where the interior is with respect to the subspace topology of $P_x$ and $\Conv(\cdot)$ is the convex hull.
\end{Lemma}

\begin{proof}
It is enough to show there are no Type~\rom{1}, Type~\rom{2}, or Type~\rom{3} points in $\Int(\Conv(S_x(2/\sqrt{3}))).$
By hypothesis, $B_x(4/\sqrt{3})$ contains no ends.  The Type~\rom{1} points are equidistant from $x$ and an end.  As there are no ends contained in $B_x(4/\sqrt{3})$, there are no Type~\rom{1} points in $\Int(\Conv(S_x(2/\sqrt{3})))$.  

By hypothesis, $x$ is in $B(R-2/\sqrt{3})$. Therefore there are no points on the boundary of $B(R)$ in $\Int(\Conv(S_x(2/\sqrt{3})))$ and therefore no Type~\rom{2} points  in $\Int(\Conv(S_x(2/\sqrt{3})))$. 

 As a Type~\rom{3} point is equidistant from three or more axes, at some distance $\ell$, it is the center of a ball tangent to three unit balls.  This is because a capped $t$-cylinder contains a unit ball which meets the ball of radius $\ell$ centered at the Type~\rom{3} point.  These balls do not overlap as the interiors of the capped $t$-cylinders have empty intersection.  Lemma $3$ of \cite{kuperberg1991placing} states that if a ball of radius $\ell$ intersects three non-overlapping unit balls in $\mathbb{R}^3$, then  $\ell \ge 2/\sqrt{3} -1.$ It follows that there are no Type~\rom{3} points in $\Int(\Conv(S_x(2/\sqrt{3})))$.  
 \end{proof}

\begin{Lemma}
Fix a packing $\mathscr{C}$.  Then for for all $i\ne j$ and $i,j \in I^*\negthinspace,$ there is a supporting hyperplane $Q$ of $\Int(C_i)$ that is parallel to $a_i$ and separating $\Int(C_i \cap L_i)$ from $\Int(C_j \cap L_i).$
\end{Lemma}
\begin{proof}

We can extend $C_i\cap L_i$ to an infinite cylinder $\bar C_i$ where $C_j \cap L_i$ and $\bar C_i$ have disjoint interiors. The sets $C_j \cap L_i$ and $\bar C_i$ are convex, so the Minkowski hyperplane separation theorem gives the existence of a hyperplane separating $\Int(C_j \cap L_i)$ and $\Int(\bar C_i)$.  This hyperplane is parallel to the axis $a_i$ by construction.  We may take $Q$ to be the parallel translation to a supporting hyperplane of $\Int(C_i)$ that still separates  $\Int(C_i \cap L_i)$ from $\Int(C_j \cap L_i).$  See Figure \ref{seper} for an example.
\end{proof}

\begin{figure}[h!]
  \centering
      \includegraphics[width=0.50\textwidth]{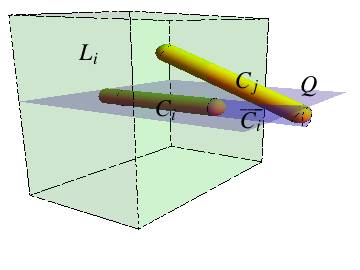}
       \caption{The hyperplane $Q$ separates $\Int(C_i \cap L_i)$ from $\Int(C_j \cap L_i).$}
   \label{seper}
\end{figure} 

\begin{Lemma} \label{key2}
Fix a packing $\mathscr{C}.$ Fix a point $x$ on the axis $a_i$ of $C_i$ such that $B_x(4/\sqrt{3})$ contains no ends.
Let $y$ and $z$ be points on the circle $S_x(2/\sqrt{3})$.  If each of $y$ and $z$ is equidistant from $C_i$ and $C_j$, then the angle $yxz$ is smaller than or equal to $2 \arccos (\sqrt{3}-1)  := \alpha_0,$ which is approximately $85.88^\circ.$
\end{Lemma}
\begin{proof}
 By hypothesis, $B_x(4/\sqrt{3})$ contains no ends, including the end of the axis $a_i$. Therefore any points of $C_j$ that are not in $L_i$ are at a distance greater than $4/\sqrt{3}$ from $x$.  The points of $C_i$ and $C_j$ that $y$ and $z$ are equidistant from must be in the slab $L_i$, so it is enough to consider $y$ and $z$ equidistant from $C_i$ and $C_j \cap L_i.$ 
 
 By construction, the hyperplane $Q$ separates all points of $C_j \cap L_i$ from $x$. Let $k$ be the line of intersection between $P_x$ and $Q$.  As $y$ and $z$ are at a distance of $2/\sqrt{3} -1$ from both  $C_j\cap L_i$ and $C_i$, they are at most that distance from $Q$.  They are also at most that distance from $k$. The largest possible angle $yxz$ occurs when $y$ and $z$ are on the $x$ side of $k$ in $P_x$, each at exactly the distance $2/\sqrt{3} -1$ from $k$ as illustrated in Figure \ref{fig7} .  This angle is exactly $2 \arccos (\sqrt{3}-1) := \alpha_0.$
\end{proof}

\begin{figure}[h!]
  \centering
      \includegraphics[width=0.50\textwidth]{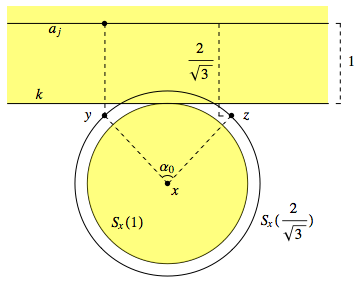}
  \caption{An extremal configuration for the angle $\alpha_0.$}
 \label{fig7}
\end{figure} 

The following lemma is proved in \cite{bezdek1990maximum}.

\begin{Lemma}\label{key3}
Let $y$ and $z$ be points on $S_x(2/\sqrt{3})$ such that $60^\circ < yxz < \alpha_0.$  For every parabola $p$ passing through $y$ and $z$ and having $S_x(1)$ on its convex side, let $xypzx$ denote the region bounded by segments $xy$, $xz$, and the parabola $p$. Let $p_0$ denote the parabola passing through $y$ and $z$ and tangent to $S_x(1)$ at its apex.
$$\Area(xyp_0zx) \le \Area(xypzx).$$
\end{Lemma}

\subsection{Truncating and rearranging}
Consider the Dirichlet slice $d_x$ of a point $x$ satisfying the conditions of Proposition \ref{qual}.  The following steps produce a region with no greater area than that of $d_x$.

\emph{Step 1}: The boundary of $d_x$ intersects $S_x(2/\sqrt{3})$ in a finite number of points.  Label them $y_1, y_2, \dots y_n, y_{n+1} = y_1$ in clockwise order.  Intersect $d_x$ with $S_x(2/\sqrt{3})$ and call the new region $d_x^*$.

By Lemma \ref{key1}, this is a region bounded by arcs of $S_x(2/\sqrt{3})$, parabolic arcs and line segments, with non-analytic points on $S_x(2/\sqrt{3})$. 

\emph{Step 2}: For $i = 1, 2, \dots , n$ if $y_ixy_{i+1} > 60^\circ$ and if  the boundary of $d_x^*$ between $y_i$ and $y_{i+1}$  is a circular arc of $S_x(2/\sqrt{3})$, then introduce additional vertices $z_{i_1}, z_{i_2}, \dots z_{i_k}$ on the circular arc $y_iy_{i+1}$ so that the polygonal arc $y_iz_{i_1} z_{i_2} \dots z_{i_k}y_{i+1}$ does not intersect $S_x(1)$.  Relabel the set of vertices $v_1, v_2, \dots v_m, v_{m+1} = v_1$ in clockwise order.

\emph{Step 3}: If $v_ixv_{i+1} \le 60^\circ$ then truncate $d_x^*$ along the line segment $v_iv_{i+1}$ keeping the part of $d_x^*$ which contains $S_x(1)$. This does not increase area by construction. If $v_ixv_{i+1} > 60^\circ$ then $v_iv_{i+1}$ is a parabolic arc.  Replace it by the parabolic arc through $v_i$ and $v_{i+1}$ touching $S_x(1)$ at its apex. This does not increase area by Lemma \ref{key2}.  This new region $d_x^{**}$ has smaller area than $d_x,$ contains $S_x(1),$ and bounded by line segments and parabolic arcs touching $S_x(1)$ at their apexes, with all non-analytic points of the boundary on $S_x(2/\sqrt{3})$.  If consecutive non-analytic points on the boundary have interior angle no greater than $60^\circ$, they are joined by line segments.  If they have interior angle between $60^\circ\negthinspace$ and $\alpha_0$, they are joined by a parabolic arc.  

The following lemma is a consequence of \cite[\S 6]{bezdek1990maximum}, which determines the minimum area of such a region.

\begin{Lemma}
The region $d_x^{**}$ has area at least $\sqrt{12}.$
\end{Lemma}

Proposition \ref{qual} follows.

\section{Decomposition of $B(R)$ and Density Calculation}\label{part2}
\subsection{Decomposition}
Fix a packing $\mathscr{C}$.  Fix $R \ge 2/\sqrt{3}$ and restrict to $\mathscr{C}^*\hspace{-3pt}.$
Let the set $A$ be the union of the axes $a_i$ over $I^*\hspace{-3pt}.$  Let $\mathrm{d}\mu$ be the 1-dimensional Hausdorff measure on $A$.  Let $X$ be the subset of qualified points of $A$.  Let $Y$ be the subset of $A$ given by $\{x \in A : B_x(\frac{4}{\sqrt{3}}) \textrm{ contains no ends}\}.$  Let $Z$ be the subset of $A$ given by $\{x\in A : B_x(\frac{4}{\sqrt{3}}) \textrm{ contains an end}\}.$  Notice that $Y\subseteq X \subseteq A$ from Proposition \ref{qual} and $Z = A-Y$ by definition.  

The sets are $A$, $X$, $Y$, and $Z$ are measurable.  The set $A$ is just a finite disjoint union of lines in $\mathbb{R}^3$.  The area of the Dirichlet slice $d_x$ is piecewise continuous on $A$, so $X$ is a Borel subset of $A$. Similarly the conditions defining $Y$ and $Z$ make them Borel subsets of $A$.  The ball $B(R)$ is finite volume, so $I^*$ has some finite cardinality $n$.

Decompose $B(R)$ into the regions $\{D^0_i\}_{i=1}^n$, $\{D^1_i\}_{i=1}^n$ and $\{D^2_i\}_{i=1}^n$.  Further decompose the regions $\{D^0_i\}_{i=1}^n$ into Dirichlet slices $d_x$, where $x$ is an element of $A.$

\subsection{Density computation}
From the definition of density,

\begin{align*}
&\rho(\mathscr{C},R-2/\sqrt{3},R) = 
\frac{\sum_{I^*}\Vol(C_i^0)+\sum_{I^*}\Vol(C_i^1)+\sum_{I^*}\Vol(C_i^2)}{\sum_{I^*}\Vol(D_i^0)+\sum_{I^*}\Vol(D_i^1)+\sum_{I^*}\Vol(D_i^2)}
\end{align*}

as $C_i^j \subseteq D_i^j$, and $\Vol(C_i^0) = t\pi$, and $\Vol(C_i^1)=\Vol(C_i^2) = \frac{2}{3}\pi$, it follows that 

\begin{equation}\label{inequality}
\rho(\mathscr{C},R-2/\sqrt{3},R) \le \frac{nt\pi+n\frac{4}{3}\pi}{\sum_{I^*}\Vol(D_i^0)+n\frac{4}{3}\pi}.
\end{equation}

Finally, $\rho(\mathscr{C},R-2/\sqrt{3},R)$ is explicitly bounded by the following lemma.

\begin{Lemma}\label{bound} For $t \ge 2t_0$,
$$\sum_{I^*} \Vol(D_i^0) \ge n(\sqrt{12}(t-2t_0) + \pi(2t_0)).$$
\end{Lemma}

\begin{proof}
The sum $\sum_{I^*}\hspace{-2pt}\Vol(D_i^0)$ may be written as an integral of the area of the Dirichlet slices $d_x $ over $A$

\begin{equation*}\sum_{I^*} \Vol(D_i^0) \hspace{2pt}= \int\limits_A \Area(d_x) \,\mathrm{d}\mu. \end{equation*}

Using the area estimates from Proposition \ref{qual}, there is an inequality

\begin{equation*}\int\limits_A \Area(d_x) \,\mathrm{d}\mu \ge \int\limits_{X} \sqrt{12}\,\mathrm{d}\mu +\hspace{-7pt} \int\limits_{A-X}\hspace{-4pt} \pi \,\mathrm{d}\mu.\end{equation*}

As $\sqrt{12} > \pi$ and the integration is over a region $A$ with $\mu(A)< \infty$, passing to the subset $Y \subseteq X$ gives

$$\int\limits_{X} \sqrt{12}\,\mathrm{d}\mu +\hspace{-7pt}  \int\limits_{A-X} \hspace{-4pt}\pi \,\mathrm{d}\mu \ge \int\limits_{Y} \sqrt{12}\,\mathrm{d}\mu +\hspace{-7pt}  \int\limits_{A-Y} \hspace{-4pt}\pi \,\mathrm{d}\mu \hspace{3pt}=\hspace{-2pt} \int\limits_{A-Z} \sqrt{12}\,\mathrm{d}\mu+\hspace{-4pt}  \int\limits_{Z} \pi\,\mathrm{d}\mu.$$

The measure of $Z$ is the measure of the subset of $A$ that is contained in all the balls of radius $4/\sqrt{3}$ about all the ends of all the cylinders in the packing.  This is bounded from above by considering the volume of cylinders contained in balls of radius $4/\sqrt{3}+1.$  If the cylinders completely filled the ball, they would contain at most axis length $\frac{4}{3}(\frac{4}{\sqrt{3}}+1)^3 = t_0$. As each cylinder has two ends, there are at worst $2n$ disjoint balls to consider.  Therefore $2nt_0 \ge \mu(Z).$ 

Provided $t\ge 2t_0$, we have the inequality

$$\int\limits_{A-Z} \sqrt{12}\,\mathrm{d}\mu+\hspace{-4pt}  \int\limits_{Z} \pi\,\mathrm{d}\mu \ge (nt-2nt_0)\sqrt{12} + 2n(t_0)\pi.$$
\end{proof}

By combining inequality (\ref{inequality}) with Lemma \ref{bound} and simplifying, it follows that 
\begin{align*}
\rho(\mathscr{C},R-2/\sqrt{3},R) \le \frac{t+\frac{4}{3}}{\frac{\sqrt{12}}{\pi}(t-2t_0) + (2t_0)+\frac{4}{3}}
\end{align*}

for an arbitrary packing $\mathscr{C}$ of $\mathbb{R}^3$ by capped congruent $t$-cylinders.  
\section{Conclusions, applications, further questions}

\subsection{A rule of thumb, a dominating hyperbola}
For $t \ge 0$, the upper bounds for the density of packings by capped and uncapped $t$-cylinders are dominated by curves of the form $\frac{\pi}{\sqrt{12}} + N/t$.  Numerically, one finds a useful rule of thumb:

\begin{theorem}
 The upper density $\rho^+$ of a packing $\mathscr{C}$ of $\mathbb{R}^3$ by capped $t$-cylinders satisfies
$$\rho^+(\mathscr{C}) \le \frac{\pi}{\sqrt{12}} + \frac{10}{t}.$$
\end{theorem}

\begin{theorem}
 The upper density $\rho^+$ of a packing $\mathscr{C}$ of $\mathbb{R}^3$ by $t$-cylinders satisfies
$$\rho^+(\mathscr{C}) \le \frac{\pi}{\sqrt{12}} + \frac{10}{t}.$$
\end{theorem}

\subsection{Examples}
While the requirement that $t$ be greater than $2t_0$ for a non-trivial upper bound is inconvenient,
the upper bound converges rapidly to $\pi/\sqrt{12} = 0.9069\dots$ and is nontrivial for tangible objects, as illustrated in the table below.
\\\\
 \begin{tabular}{cccccccc}
&Item &  Length & Diameter & t & Density $\le$& \\ \midrule \hline 
&Broomstick & 1371.6mm & 25.4mm & 108 & 0.9956...&\\ \hline
&20' PVC Pipe & 6096mm & 38.1mm & 320& 0.9353...& \\ \hline
&Capellini & 300mm & 1mm & 600& 0.9219... &\\\hline
&Carbon Nanotube & -& -&$2.64 \times 10^8$ \cite{wang2009fabrication} & 0.9069... & \\  \hline
\end{tabular}
\\\\
\subsection{Some further results}
There are other conclusions to be drawn.  For example: \hspace{-2pt}Consider the density of a packing $\mathscr{C} = \{C_i\}_{i \in I}$ of $\mathbb{R}^3$ by congruent unit radius circular cylinders $C_i$, possibly of infinite length.  The upper density $\gamma^+$ of such a packing may be written

$$ \gamma^+(\mathscr{C}) = \limsup_{r \rightarrow \infty} \sum_I\frac{\Vol(C_i\cap B_0(r))}{\Vol(B_0(r))} $$

and coincides with $\rho^+(\mathscr{C})$ when the lengths of $C_i$ are uniformly bounded. 
 \begin{theorem}
The upper density $\gamma$ of half-infinite cylinders is exactly $\pi/\sqrt{12}.$
 \end{theorem}
 
 \begin{figure}[h!]
  \centering
      \includegraphics[width=0.50\textwidth]{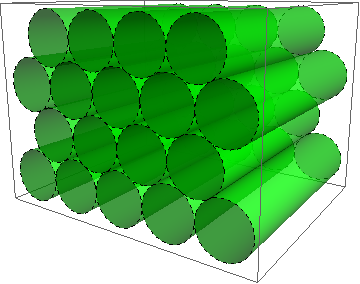}
  \caption{Packing cylinders with high density.}
  \label{obvi}
\end{figure}

 \begin{proof}
 The lower bound is given by the obvious packing $\mathscr{C}'$ with parallel axes (Figure \ref{obvi}) and $\gamma^+(\mathscr{C}') =\pi/\sqrt{12}$. Since a packing $\mathscr{C}(\infty)$ of $\mathbb{R}^3$ by half-infinite cylinders also gives a packing $\mathscr{C}(t)$ of $\mathbb{R}^3$ by $t$-cylinders, the inequality
 $$\frac{t}{\frac{\sqrt{12}}{\pi}(t-2-2t_0) + (2t_0)+\frac{4}{3}} \ge \rho^+(\mathscr{C}(t)) = \gamma^+(\mathscr{C}(t)) \ge \gamma^+(\mathscr{C}(\infty))$$
holds for all $t \ge 2t_0.$
 \end{proof}

 \begin{theorem}
 Given a packing $\mathscr{C} = \{C_i\}_{i \in I}$ by non-congruent capped unit cylinders with lengths constrained to be between $2t_0$ and some uniform upper bound $M$, the density satisfies the inequality  $$\rho^+(\mathscr{C}) \le \frac{t+\frac{4}{3}}{\frac{\sqrt{12}}{\pi}(t-2t_0) + (2t_0)+\frac{4}{3}}$$  where $t$ is the average cylinder length given by  $\liminf_{r \rightarrow \infty} \mu(a_i) / |J|$, where $J$ is the set $\{i\in I: C_i \subseteq B(r)\}$.
 \end{theorem}
 
  \begin{proof}
None of the qualification conditions require a uniform length $t$.  Inequality \ref{inequality} may be considered with respect to the total length of $A$ rather than $nt$. 
 \end{proof}
 It may be easier to compute a bound using the following
 
  \begin{corollary}
 Given a packing $\mathscr{C} = \{C_i\}_{i \in I}$ by non-congruent capped unit cylinders with lengths constrained to be between $2t_0$ and some uniform upper bound $M$, the density satisfies the inequality  $$\rho^+(\mathscr{C}) \le \frac{t+\frac{4}{3}}{\frac{\sqrt{12}}{\pi}(t-2t_0) + (2t_0)+\frac{4}{3}}$$  where $t$ is the infimum of cylinder length.
 \end{corollary}
 
\begin{proof}
The stated bound is a decreasing function in $t$, so this follows from the previous theorem.
\end{proof}
 
\subsection{Questions and Remarks}
Similar but much weaker results can be obtained for the packing density of curved tubes by realizing them as containers for cylinders.  Better bounds on tubes would come from better bounds on $t$-cylinders for $t$ small.  There is the conjecture of Wilker, where the expected packing density of congruent unit radius circular cylinders of $any$ length is $exactly$ the planar packing density of the circle, but that is certainly beyond the techniques of this paper. A more tractable extension of this might be to parametrize the densities for capped $t$-cylinders from the sphere to the infinite cylinder by controlling the ends.  In this paper, the analysis assumes anything in a neighborhood of an end packs with density 1, whereas it is expected that the ends and nearby sections of tubes would pack with density closer to $\pi/\sqrt{18}.$ $\hspace{3pt}$  In \cite{torquato2012organizing}, it is conjectured that the densest packing is obtained from extending a dense sphere packing, giving a density bound of

\[ \rho^+(\mathscr{C}(t)) =\frac{\pi}{\sqrt{12}}  \hspace{4pt} \frac{t + \frac{4}{3} } {t + \frac {2\sqrt{6}}{3}  }.\]\\

The structure of high density cylinder packings is also unclear.  For infinite circular cylinders, there are nonparallel packings with positive density \cite{kuperberg1990nonparallel}.  In the case of finite length $t$-cylinders, there exist nonparallel packings with density close to $\pi/\sqrt{12}$, obtained by laminating large uniform cubes packed with parallel cylinders, shrinking the cylinders and perturbing their axes.  It is unclear how or if the alignment of cylinders is correlated with density.  Finally, as the upper bound presented is not sharp, it is still not useable to control the defects of packings achieving the maximal density.  A conjecture is that, for a packing of $\mathbb{R}^3$ by $t$-cylinders to achieve a density of $\pi/\sqrt{12}$, the packing must contain arbitrarily large regions of $t$-cylinders with axes arbitrarily close to parallel.

\bibliography{RodPaperArXiv}
\end{document}